\documentclass[a4paper,11pt]{article}

\usepackage{amsmath, amscd, amsfonts}

\newtheorem{theorem}{Theorem}

\newtheorem{lemma}[theorem]{Lemma}

\newtheorem{conjecture}[theorem]{Conjecture}

\newcommand{\qed}{\rule{7pt}{7pt}}

\newenvironment{proof}{\noindent{\bf Proof}\hspace*{1em}}{\qed\bigskip}
\newenvironment{proof-sketch}{\noindent{\bf Sketch of Proof}\hspace*{1em}}{\qed\bigskip}
\newenvironment{proof-idea}{\noindent{\bf Proof Idea}\hspace*{1em}}{\qed\bigskip}
\newenvironment{proof-of-lemma}[1]{\noindent{\bf Proof of Lemma #1}\hspace*{1em}}{\qed\bigskip}
\newenvironment{proof-attempt}{\noindent{\bf Proof Attempt}\hspace*{1em}}{\qed\bigskip}

\title{Wilf conjecture}
\author{Junkyu An}
\date{14 July 2008}

\begin{document}
\maketitle

\begin{abstract}
Let $S(n,k)$ be the Stirling number of the second kind. Wilf conjectured that the alternating sum of $S(n,k)$
for $0\le k\le n$ is not zero for all $n>2$. In this paper, we prove that Wilf conjecture is true except at
most one number with the properties of weighted Motzkin number.
\end{abstract}

\section{Introduction}
Let $S(n,k)$ be the Stirling number of the second kind(i.e. the number of partitioning $[n]$ into $k$ nonempty subsets). For more information, 
see $\cite{EC}$. The complementary Bell numbers are $f(n)=\sum_{k=0}^{n}(-1)^k S(n,k)$. The first $f(n)$(Sloane$\cite{S}$'s $A000587$) 
for $n=0,1,2,3,4,\cdots$ are 
$$1,-1,0,1,1,-2,-9,-9,50,267,413,-2180,-17731,-50533,\cdots$$
The following conjecture comes from Wilf$\cite{EO}$.
\begin{conjecture}$\cite{EO}$
$f(n)\neq 0$ for all $n>2$.
\end{conjecture}
Wannemacker, Laffey and Osburn$\cite{W}$ showed that $f(n)\neq 0$ for all $n\not\equiv2,\;2944838(mod\;3145728)$ by
using the generating function of $f(n)$. \\
The main result of this paper is the following.
\begin{theorem}
There is at most one $n>2$ such that $f(n)=0$.
\end{theorem}
In section $2.1$, we define weighted Motzkin numbers. Section $2.2$ deals some properties of $f(n)$ by using the
properties of weighted Motzkin numbers. In section $3$, we finally prove $Theorem\;2$.

\section{Weighted Motzkin number}
\subsection{Definition}
A Motzkin path $P$ with length $n$ is a path from $(0,0)$ to $(n,0)$ consisting of steps $(1,1)$(a rise step), $(1,-1)$(a fall step) and 
$(1,0)$(a level step) that lies above the $x$-axis. It can be expressed by $p_0,p_1,\cdots,p_{n}$, a sequence of points in 
$(N\cup\{0\}) \times (N\cup\{0\})$ where \\
$(1)$ $p_0=(0,0)$, $p_{n}=(n,0)$ \\
$(2)$ $p_{i+1}-p_{i}=(1,1)$, $(1,-1)$, or $(1,0)$ \\
Let $b_x$(respectively, $c_x$ and $d_x$) be the given weight function from $N\cup\{0\}$ to $Z$. The weight of a rise step from $(x,y)$ to $(x+1,y+1)$ 
is $b_y$(respectively, the weight of a fall step from $(x,y+1)$ to $(x+1,y)$ is $d_y$ and the weight of a level step from $(x,y)$ to $(x+1,y)$ is $c_y$). 
Then, the weight of a Motzkin path $P$
(i.e. $w(P)$) is defined by the product of the weight of steps. See Figure~\ref{white} for an example. \\
\begin{figure}[!hbp]
\setlength{\unitlength}{4cm}
\begin{picture}(1,1)
\thicklines
\put(0,0){\line(0,1){1}}
\put(0,0){\line(1,0){3}}
\put(0,0){\line(1,1){0.25}}
\put(0.25,0.25){\line(1,0){0.25}}
\put(0.5,0.25){\line(1,1){0.25}}
\put(0.75,0.5){\line(1,1){0.25}}
\put(1,0.75){\line(1,0){0.25}}
\put(1.25,0.75){\line(1,-1){0.25}}
\put(1.5,0.5){\line(1,0){0.25}}
\put(1.75,0.5){\line(1,0){0.25}}
\put(2.00,0.5){\line(1,-1){0.25}}
\put(2.25,0.25){\line(1,0){0.25}}
\put(2.5,0.25){\line(1,-1){0.25}}
\put(0.06,0.15){$b_0$}
\put(0.35,0.27){$c_1$}
\put(0.56,0.40){$b_1$}
\put(0.81,0.65){$b_2$}
\put(1.1,0.77){$c_3$}
\put(1.35,0.66){$d_2$}
\put(1.6,0.52){$c_2$}
\put(1.85,0.52){$c_2$}
\put(2.1,0.41){$d_1$}
\put(2.35,0.27){$c_1$}
\put(2.60,0.16){$d_0$}
\end{picture}
\caption{A path with weight $b_0 b_1 b_2 c_1^2 c_2^2 c_3 d_0 d_1 d_2$. \label{white}}
\end{figure}
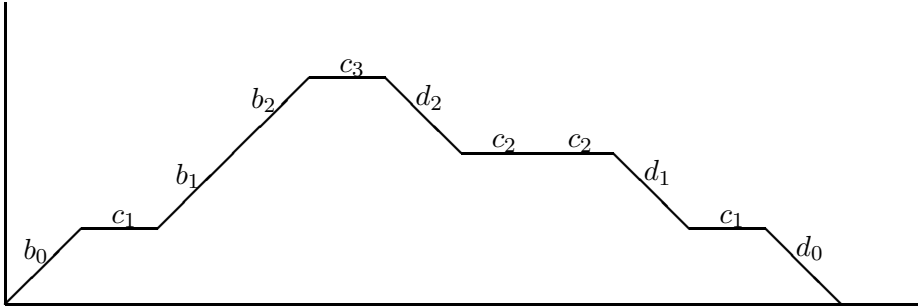
\\
The corresponding $n$th weighted Motzkin number $M_n^{b,c,d}$ is given by
\begin{equation}
M_n^{b,c,d}=\sum w(P)
\end{equation}
where the sum is over all Motzkin paths from $(0,0)$ to $(n,0)$. \\
From $\cite[Chapter\;5]{GJ}$, we know that the generating function of $M_n^{b,c,d}$ is
\begin{equation}
\sum_{n\ge 0} M_n^{b,c,d} x^n =\frac{1}{1-c_0 x-\frac{b_0 d_0 x^2}{1-c_1 x-\frac{b_1 d_1 x^2}{1-c_2 x-\frac{b_2 d_2 x^2}{\cdots}}}} \label{cf1}
\end{equation}
Since $M_n^{b,c,d}=M_n^{bd,c,1}$, we do not concern about $d_x$ in most case(assume that $d_x=1$) unless it is mentioned.
For $b_x=1$ and $c_x=1$, $M_n^{b,c}$ is the $n$th Motzkin number.
\begin{theorem}$\cite[Theorem\;2]{F}$
$$\sum_{k=0}^{n} S(n,k)u^k=M_n^{b',c'}$$
where $b'_x=u(x+1)$ and $c'_x=u+x$.
\end{theorem}
Flajolet$\cite{F}$ proved the above result by using Path diagrammes. He constructed a bijection between set partitions and weighted
Motzkin paths and he made a generalization of Francon-Viennot decomposition in $\cite{GJ}$. We remark that a lot of combinatorial counting
including $B_n$(Bell numbers), $I_n$(the number of involutions on $[n]$) can be expressed by weighted
Motzkin numbers with his result, too. In particular, we have
\begin{equation}
f(n)=M_n^{b',c'}
\end{equation}
for $b'_x=-x-1$ and $c'_x=x-1$.
From \eqref{cf1},
\begin{equation}
f(n)=M_n^{b,c,d}
\end{equation}
for
\begin{equation}
\left\{
\begin{array}{cccc}
b_x=(-x-1)/2  &, c_x=x-1 & \text{and } d_x=2 & \text{if $x$ is odd}  \\
b_x=-x-1      &, c_x=x-1 & \text{and } d_x=1 & \text{if $x$ is even}
\end{array} \right.
\end{equation}

\subsection{Some properties of $f(n)$}
Let $W_{n,k}$ be the sum of weighted paths from $(0,0)$ to $(n,k)$ that lies above the $x$-axis. By definition,
$W_{n,0}=f(n)$. We have
\begin{equation}
W_{n+1,k+1}=W_{n,k}\cdot b_k + W_{n,k+1}\cdot c_{k+1} + W_{n,k+2}\cdot d_{k+1} \label{r}
\end{equation}
for all $n,k\ge 0$ and $W_{n+1,0}=W_{n,0}\cdot c_{0} + W_{n,1}\cdot d_{0}$. \\
Let $A_r$(for $r\ge 0$) be the following $(r+1)\times (r+1)$ matrices:
\begin{equation}
\mathbf{A_r} = \left(
\begin{array}{ccccc}
c_0    & d_0    & 0      & \ldots & \ldots \\
b_0    & c_1    & d_1    & \ldots & \ldots \\
\vdots & \vdots & \ddots & \ldots & \ldots \\
\vdots & \vdots & b_{r-2}& c_{r-1}& d_{r-1}\\
\vdots & \vdots & 0      & b_{r-1}& c_{r}
\end{array} \right) \label{m}
\end{equation}
Then, $(W_{n+k,0},W_{n+k,1},\cdots,W_{n+k,r-1})\equiv A_{r-1}^k (W_{n,0},W_{n,1},\cdots W_{n,r-1})$ \\
$(mod\; b_0b_1\cdots b_{r-1})$ because $W_{n,k}\equiv 0(mod\;b_0b_1\cdots b_{r-1})$ for all $k\ge r$. \\
\begin{equation}
\mathbf{A_{4k-1}} \equiv \left(
\begin{array}{ccccc}
A          & \textbf{0} & \ldots     & \ldots     & \ldots     \\
\textbf{0} & A          & \ldots     & \ldots     & \ldots     \\
\vdots     & \vdots     & \ddots     & \ldots     & \ldots     \\
\vdots     & \vdots     & \vdots     & A          & \textbf{0} \\
\vdots     & \vdots     & \vdots     & \textbf{0} & A
\end{array} \right) (mod\; 2)
\end{equation}
where
\begin{equation}
\mathbf{A} = \left(
\begin{array}{ccccc}
1      & 1      & 0     & 0   \\
1      & 0      & 0     & 0   \\
0      & 1      & 1     & 1   \\
0      & 0      & 1     & 0
\end{array} \right)
\end{equation}
Since $A^{6}\equiv I(mod\; 2)$, $A_{4k-1}^{6}\equiv I(mod\; 2)$ where $I$ is an identity matrix. \\
Now, let $S$ be the shift operator(i.e. $S(W_{n,k})=W_{n+1,k}$ and $S(f(n))=f(n+1)$).
$(\sum_{i=0}^{t} a_i S^{i})(W_{n,k})$
means $\sum_{i=0}^{t} a_i S^{i}(W_{n,k})=\sum_{i=0}^{t} a_i W_{n+i,k}$ where $a_i\in Z$ for $1\le i\le t$.
\begin{lemma}
$$(E-1)^r (W_{n,k}) \equiv 0(mod\;2^{r})$$
for all $r\ge 1$ where $E=S^{6(2t-1)}$($t\in N$).
\end{lemma}
\begin{proof}
For given $r$, $(E-1)^r (W_{n,k}) \equiv 0(mod\;2^{r})$ for all $k> 4r-1$ because $W_{n,k}\equiv 0(mod\;2^r)$
for all $k> 4r-1$. Therefore, we only need to show for $k\le 4r-1$ and we can assume that $S=A_{4r-1}$. \\
It is proved by mathematical induction on $r$. The statement is true for $r=1$ from the fact that
$A_{3}^{6}\equiv I(mod\; 2)$. We assume that the statement is true for $r=m$$(m\ge1)$, and prove for $r=m+1$.
Using the assumption $r=m$ (i.e. $(E-1)^m (W_{n,k}) \equiv 0(mod\; 2^m)$), it is easy to check that
\begin{eqnarray*}
& & (\frac{(E-1)^m W_{n+s,0}}{2^m},\cdots,\frac{(E-1)^m W_{n+s,4m+3}}{2^m}) \\
& & \equiv {A}_{4m+3}^s (\frac{(E-1)^m W_{n,0}}{2^m},\cdots,\frac{(E-1)^m W_{n,4m+3}}{2^m})(mod\;2) \label{e2}
\end{eqnarray*}
because
$$(W_{n+s,0},\cdots,W_{n+s,4m+3})\equiv {A}_{4m+3}^s (W_{n,0},\cdots,W_{n,4m+3})(mod\;2^{m+1}) $$
We have $\frac{(E-1)^m E(W_{n,k})}{2^m}\equiv \frac{(E-1)^m (W_{n,k})}{2^m} (mod\;2)$. Therefore,
$(E-1)^{m+1} (W_{n,k}) \equiv 0(mod\;2^{m+1})$ for all $k\le 4m+3$. The proof is done.
\end{proof} \\
We remark that if $g(x)\in Z[x]$ and $(x-1)^r$ divides $g(x)$, then $g(E)\equiv 0(mod\; 2^r)$ for any nonnegative
integer $r$. Therefore, we get the following lemma.
\begin{lemma}
$$(E^{2^{k}}-1)^2(f(n))\equiv 0(mod\;2^{2k+2})$$
for all $k\ge 0$.
\end{lemma}
\begin{proof}
It is true when $k=0$. Since $E^{2^{k+1}}-1=(E-1)(E+1)(E^2+1)\cdots(E^{2^{k}}+1)$ and $E^{2^{s}}+1=(E-1)g_s(E)+2$ for all $s\ge 0$ where 
$g_s(E)=E^{2^{s}-1}+\cdots+E+1$, it is obvious from $Lemma\;4$.
\end{proof} \\
Similar to $Lemma\; 5$, it can be proved that $(E^{2^{k}}-1)(f(n))\equiv 0(mod\;2^{k+1})$. This implies that
\begin{equation}
f(n+3\cdot 2^{k+1})\equiv f(n)(mod\;2^{k+1})\;for\; all\; n \label{wc4}
\end{equation}

\section{The main result}

Now, we will prove $Theorem \;2$. \\
\begin{proof}
From $Lemma\;5$, we have
\begin{equation}
(E^{2^{k+1}}-1)(f(n)) \equiv 2(E^{2^{k}}-1)(f(n))(mod\;2^{2k+2}) \label{wc1}
\end{equation}
for all $k\ge 0$. This implies that $f(n+3\cdot 2^{k+2}\cdot (2t-1))-f(n)\equiv 2(f(n+3\cdot 2^{k+1}
\cdot (2t-1))-f(n))(mod\;2^{2k+2}))$ for $k\ge 0$ and $t\ge 1$.\\
Now, let's show that
\begin{equation}
f(n)\not\equiv 0(mod\;2^{k+2})\;for\; all\; n\not\equiv 2, a_k(mod\;3\cdot 2^k) \label{wc3}
\end{equation}
for $k\ge 5$ and some $2<a_k<3\cdot 2^k$ such that $a_k\equiv 38 (mod\;3\cdot 2^5)$. \\
It is proved by mathematical induction on $k$. Wannemacker, Laffey and Osburn$\cite{W}$ showed that $f(n)\not\equiv 0(mod\;2^{7})$
for all $n\not\equiv 2,38(mod\;3\cdot 2^5)$. Therefore, the statement is true for $k=5$ and $a_5=38$.
We assume that the
statement is true for $r=m$$(m\ge 5)$, and prove for $r=m+1$. Using the assumption $r=m$, we have
$f(n)\not\equiv 0(mod\;2^{m+3})$ for all $n\not\equiv 2,2+3\cdot 2^m,a_m,a_m+3\cdot 2^m(mod\;3\cdot 2^{m+1})$. \\
If there exist some $a$ and $b$ such that $f(2+3\cdot 2^m+3\cdot 2^{m+1} a)\equiv f(2+3\cdot 2^{m+1}b)\equiv 0(mod\;2^{m+3})$,
let
\begin{eqnarray*}
A & = & 2+3\cdot 2^m+3\cdot 2^{m+1} a \\
B & = & 2+3\cdot 2^{m+1}b \\
C & = & \frac{A+B}{2} = 2+3\cdot 2^{m-1}+3\cdot 2^m(a+b)
\end{eqnarray*}
(if $a<b$, change $a$ into $a+4b$ by using \eqref{wc4})
But we know that
\begin{equation}
f(A)-f(C)\equiv 2(f(B)-f(C))(mod\;2^{m+3}) \label{wc2}
\end{equation}
from \eqref{wc1} by taking $n=2+3\cdot 2^{m+1}b$, $t=a-b+1$ and $k=m-2$(i.e. $n=2+3\cdot 2^{m+1}b$,
$n+3\cdot 2^{k+3}\cdot (2t-1)=2+3\cdot 2^m+3\cdot 2^{m+1} a$, $n+3\cdot 2^{k+2}\cdot (2t-1)=
2+3\cdot 2^{m-1}+3\cdot 2^m(a+b)$ and $2k+2=2m-2\ge m+3$ for $m\ge5$). Therefore, $f(C)\equiv 0(mod\;2^{m+2})$ and
this contradicts the assumption $r=m$. But, we know that $f(2)=0$ and therefore, $f(n)\not\equiv 0(mod\;2^{m+3})$
for all $n\equiv 2+3\cdot 2^m(mod\;3\cdot 2^{m+1})$. Similarly, we have $f(n)\not\equiv 0(mod\;2^{m+3})$ for
all $n\equiv a_m+3\cdot 2^m(mod\;3\cdot 2^{m+1})$ or $n\equiv a_m(mod\;3\cdot 2^{m+1})$. The proof of
\eqref{wc3} is done and we can see that $a_{m+1}=a_m$ or $a_m+3\cdot 2^m$. \\
If we assume that there exist x and y such that $f(x)=f(y)=0$ and $x\neq y>2$,
then we can find some $k$ such that $x,y<3\cdot 2^k$. Therefore, $f(2)=f(x)=f(y)\equiv 0(mod\;2^{k+2})$ and this
contradicts \eqref{wc3}. This is because $2,x,y$ are different modulo $3\cdot 2^k$.
\end{proof} \\
Note that Wilf conjecture is true if we can show that $a_i$'s are increasing. In that case, if there exists
$n>2$ such that $f(n)=0$, then we can find some $k$ such that $a_k>n$. Then, $n\not\equiv a_k(mod\;3\cdot 2^k)$ and
Wilf conjecture is true. The following table shows $a_i$ for $5\le i\le 20$. \\
\begin{tabular}{|r|l|}
\hline
$a_5$  & 38 \\
$a_6$  & $134=a_5+3\cdot 2^5$ \\
$a_7$  & $326=a_6+3\cdot 2^6$ \\
$a_8$  & 326 \\
$a_9$  & 326 \\
$a_{10}$ & $1862=a_9+3\cdot 2^9$ \\
$a_{11}$ & 1862 \\
$a_{12}$ & $8006=a_{11}+3\cdot 2^{11}$ \\
$a_{13}$ & $20294=a_{12}+3\cdot 2^{12}$ \\
$a_{14}$ & $44870=a_{13}+3\cdot 2^{13}$ \\
$a_{15}$ & $94022=a_{14}+3\cdot 2^{14}$ \\
$a_{16}$ & $192326=a_{15}+3\cdot 2^{15}$ \\
$a_{17}$ & 192326 \\
$a_{18}$ & $585542=a_{17}+3\cdot 2^{17}$ \\
$a_{19}$ & $1371974=a_{18}+3\cdot 2^{18}$ \\
$a_{20}$ & $2944838=a_{19}+3\cdot 2^{19}$ \\
\hline
\end{tabular} \\
\\
\textbf{Acknowledgments} I would like to thank Alexander Postnikov for valuable comments, discussions and
help to improve the paper.

\end{document}